\documentclass[reqno]{amsart}

\usepackage{amssymb}
\usepackage[all]{xy}

\def\eu{\mathfrak}
\def\ma{\mathbb}
\def\mc{\mathcal}

\def\p{{\eu p}_{\infty}}
\def\P{{\mc P}_{\infty}}
\def\f{{\ma F}_q^{\ast}}
\def\F{{\ma F}_q}
\def\pK{\eu p}
\def\pL{\eu P}
\def\fin{\hfill\qed\bigskip}
\def\ge#1{#1_{{gex}}}
\def\g#1{#1_{ge}}
\def\*#1{#1^*}
\def\cicl#1#2{k(\Lambda_{{#1}^{#2}})}
\def\lra{\longrightarrow}

\newcommand{\sgn}{\operatorname{sgn}}
\newcommand{\Sig}{\operatorname{Sig}}
\newcommand{\ext}{\operatorname{ext}}
\newcommand{\Gal}{\operatorname{Gal}}
\newcommand{\Aut}{\operatorname{Aut}}

\newcommand{\Id}{\operatorname{Id}}
\newcommand{\im}{\operatorname{im}}
\newcommand{\N}{\operatorname{N}}
\newcommand{\rest}{\operatorname{rest}}
\newcommand{\Hom}{\operatorname{Hom}}
\newcommand{\hooklongrightarrow}{\lhook\joinrel\longrightarrow}

\newcounter{bean}
\newcounter{2bean}

\def\l{
\begin{list}
{\rm{(\alph{bean}).-}}{\usecounter{bean}
\setlength{\labelwidth}{0.8in}
\setlength{\labelsep}{0.3cm}
\setlength{\leftmargin}{1cm}}}

\def\las{\begin{list}
	{{\rm {(\arabic{2bean})}}}{\usecounter{2bean}
\setlength{\labelwidth}{0.8in}
\setlength{\labelsep}{0.3cm}
\setlength{\leftmargin}{1cm}}}

\numberwithin{equation}{section}
\newtheorem{theorem}{Theorem}[section]
\newtheorem{proposition}[theorem]{Proposition}

\newtheorem{example}[theorem]{Example}
\newtheorem{remark}[theorem]{Remark}
\newtheorem{definition}[theorem]{Definition}
\newtheorem{corollary}[theorem]{Corollary}

\title[Class fields and extended genus fields]
{Class fields, Dirichlet characters and extended genus fields of global
function fields}
 
\author[M. Rzedowski]{Martha Rzedowski--Calder\'on}
\address{Departamento de Control Autom\'atico\\
Centro de Investigaci\'on y de Estudios Avanzados del I.P.N.}
\email{mrzedowski@ctrl.cinvestav.mx}

\author[G. Villa]{Gabriel Villa--Salvador}
\address{Departamento de Control Autom\'atico\\
Centro de Investigaci\'on y de Estudios Avanzados del I.P.N.}
\email{gvillasalvador@gmail.com, gvilla@ctrl.cinvestav.mx}



\date{April 4, 2022}

\begin{document}

\begin{abstract}

We obtain the extended genus field of an
abelian extension of a rational function field. 
We follow the definition
of Angl\`es and Jaulent, which uses class
field theory. First we show that the natural 
definition of extended
genus  field of a cyclotomic function field obtained by 
means of Dirichlet characters is the same as the one
given by Angl\`es and Jaulent. Next we study the
extended genus field of a finite abelian extension of a
rational function field along the lines of the study
of genus fields of abelian extensions of
rational function fields and compare this approach
with the one given by Angl\`es and Jaulent.

\end{abstract}

\maketitle

\section{Introduction}\label{S1}

C.F. Gauss \cite{Ga1801}
was the first to study extended genus fields introducing
the genus concept in the context of
quadratic forms. It seems that D. Hilbert \cite{Hi1894}
and then E. Hecke \cite{He23} were
the first in translating Gauss' language of genus theory to the number
field setting. Later on, H. Hasse \cite{Ha51} considered
the extended genus field of a quadratic number field by
means of the kernel of characters in the context of
class field theory. Starting from here, H.W. Leopoldt
\cite{Le53} generalized Hasse's results defining
the extended genus field of a finite abelian extension
of the field of rational numbers. In his study, Leopoldt
studied the arithmetic of abelian number fields using {\em Dirichlet
characters}. A. Fr\"ohlich defined the concept of
genus field of an arbitrary number field 
\cite{Fr59-1,Fr59-2}. Fr\"ohlich's definition of the genus field
of a finite number field $K$ is $\g K:=KF$ where $F$
is the maximal abelian extension of the field of rational
numbers ${\ma Q}$ contained in the Hilbert Class Field
(HCF) $H_K$ of $K$. Similarly, the {\em extended
genus field} of a number field $K$ is defined as
$\ge K:=K F_+$, where $F_+$ is the maximal abelian
extension of ${\ma Q}$ contained in the extended or narrow
Hilbert Class Field $H_K^+$. Since then, numerous
authors have studied both, genus and extended genus
fields of number fields.

In the number field setting, the concepts of HCF $H_K$ and extended
HCF $H_K^+$ of $K$ are defined canonically as the maximal unramified
extension and as the maximal unramified extension at the
finite primes of $K$, respectively. In particular, the concepts of
genus field and of extended genus field are canonically defined.
We have  $K\subseteq \g K\subseteq
H_K$ and the Galois group $\Gal(H_K/K)$ is isomorphic
to the class group $Cl_K$ of $K$.
The genus field $\g K$ corresponds to a 
subgroup $G_K$ of $Cl_K$ and we have $\Gal(\g K/K)\cong
Cl_K/G_K$. The degree $[\g K:K]$ is called the {\em genus 
number} of $K$ and $\Gal (\g K/K)$ is called the {\em genus group}
of $K$. Similarly, $K\subseteq \ge K\subseteq H_K^+$ 
and $\ge K$ corresponds
to a subgroup $G_{K^+}$ of $\Gal(H_K^+/K)\cong Cl_{K^+}$.

The function field case is quite different since there are several
possible good definitions of HCF and consequently of extended
HCF. The definition given depends heavily on which aspect one
wants to study. The first possible definition of HCF of a
global function field $K$ would be as the maximal abelian
extension of $K$. However this extension is of infinite degree
over $K$ because it contains all the extensions of constants.
If one wants a finite extension as the HCF, it is necessary to impose
a splitting condition on some primes since every prime
is eventually inert in the extensions of constants.

The first person who considered genus fields of global function fields
was R. Clement \cite{Cl92} who defined the genus field of
a cyclic Kummer extension $K$ of a rational function field $k$ of
prime degree. She followed closely the definition given by
Hasse in \cite{Ha51}. The next step was given by S. Bae and
J.K. Koo in \cite{BaKo96} where they developed the concept of genus
fields along the lines of Fr\"ohlich's work. In fact, they defined
the extended genus field for an arbitrary global function
field $K$ using a generalization of cyclotomic function field
extensions given by the Carlitz module. The 
genus field of a general abelian extension
of a global function field over a rational
function field using Dirichlet characters can be found in 
\cite{MaRzVi2013,MaRzVi2017,BaMoReRzVi2018}.

The definition of HCF of a global function field we
will be using is the following.
For a function field $K$, let $S$ be a finite nonempty set of
places of $K$. Then the HCF relative to $S$, $H_{K,S}$ 
of $K$, is the maximal 
unramified abelian extension of $K$ such that every element
of $S$ decomposes fully in $H_{K,S}$. We have that $H_{K,S}/K$ is 
a finite extension and $\Gal(H_{K,S}/K)\cong Cl_S$, the class
group $Cl_S$ of the Dedekind domain ${\mc O}_S$ of the elements
of $K$ with poles contained in $S$.

There is no simple definition of what the extended HCF of a function
field $K$ should be. In the number field case, we have that the
extended HCF $H_K^+$ of $K$ is the maximal abelian extension of $K$
unramified at the finite primes and it is the ray class field of the
modulus $1_+:=\prod_{\text{$\pK$ is real}}\pK$. We have that
a place $\pK$ decomposes fully in $H_K^+/K$ if and only if
$\pK$ is principal generated by a totally positive element. Following
this idea, B. Angl\`es and J.-F. Jaulent \cite{AnJa2000} defined the
extended HCF of a global function field $K$ in an analogous way.
In fact the definition of Angl\`es and Jaulent works for any
global field, either numeric or function.

In this paper, we use the above definition as
the extended HCF $H_K^+$ of a global
function field $K$. Using this definition we define the {\em extended
genus field} $\ge K$ of $K$ over $k$, where $k=\F(T)$
and $K/k$ is a finite geometric extension, as $\ge K=KF$ 
where $F$ is the maximal abelian 
extension of $k$ contained in $H_K^+$.
The main purpose of this paper is to show that if
$K$ is contained in a cyclotomic function field, then $\ge K$ is
the maximal abelian extension of $K$ contained in a cyclotomic
function field unramified at the finite primes.  This is Theorem \ref{T4.4},
where we obtain that $\ge K$ is the field associated to a group
of Dirichlet characters. In particular
this definition is the same as the one given in
\cite{RaRzVi2019} in the case of a field contained in a 
cyclotomic function field but it might be different in the general case,
even for abelian extensions.

The general expression of $\ge K$ for a finite abelian extension
$K$ of $k$ is given in Theorem \ref{T5.1}.

\section{Preliminaries and notations}\label{S2}

We denote by $k=\F(T)$ the global rational function
field with field of constants the finite field of $q$ elements $\F$.
Let $R_T=\F[T]$ be the ring of polynomials, that is, the
ring of integers of $k$ with respect to the pole of $T$, the
infinite prime $\p$. Let $R_T^+:=\{P\in R_T\mid P \text{\ is
monic and irreducible}\}$. The elements of $R_T^+$ are the
{\em finite primes} of $k$ and $\p$ is the {\em infinite prime} of $k$.
For $N\in R_T$, $\Lambda_N$ denotes the $N$--th torsion of 
the Carlitz module. A finite extension $F/k$ will be called 
{\em cyclotomic} if there exists $N\in R_T$ such that
$k\subseteq F\subseteq \cicl N{}$.

Given a cyclotomic function field $E$, the group of Dirichlet characters
$X$ corresponding to $E$ is the group $X$ such that 
$X\subseteq\widehat{(R_T/\langle N\rangle)^*}\cong 
\widehat{\Gal(\cicl N{}/k)}=\Hom\big((R_T/\langle N\rangle)^*,
{\ma C}^*\big)$
and $E=\cicl N{}^{H}$ where $H=\cap_{\chi\in X}\ker \chi$.
For the basic results on Dirichlet characters, we refer to
\cite[Ch. 3]{Wa97} or \cite[\S 12.6]{Vil2006}.

For a group of Dirichlet characters $X$, let 
$Y=\prod_{P\in R_T} X_P$ where $X_P=\{\chi_P\mid 
\chi\in X\}$ and $\chi_P$ is the $P$--th component of $\chi$:
$\chi=\prod_{P\in R_T^+} \chi_P$. If $E$ is the field corresponding
to $X$, we define $\ge E$ as the field corresponding
to $Y$. We have that $\ge E$ is the maximal 
unramified extension at the finite
primes of $E$ contained in a cyclotomic function field $\cicl N{}$.
The infinite prime $\p$ might be ramified in $\ge E/k$
(see \cite{MaRzVi2013}).

Let $L_n=\cicl {1/T^{n+1}}{}^{\f}$, $n\in{\ma N}\cup\{0\}$
where $\*{\F}\subseteq\Big(R_{1/T}/\langle 1/T^{n+1}\rangle
\Big)^*$, is isomorphic to the inertia group of 
the prime corresponding to $T$ in $k\big(\Lambda_{
1/T^{n+1}}\big)/k$. The prime
$\p$ is the only ramified prime in $L_n/k$ and it is totally and
wildly ramified. For $n\in {\ma N}\cup\{0\}$ and any finite extension $F/k$,
we denote ${_nF}:= L_n F$. 
For $m\in {\ma N}$, and for any finite extension
$F/k$, $F_m$ denotes the extension of constants: $F_m=F
{\ma F}_{q^m}$. In particular $k_m={\ma F}_{q^m}(T)$.

Given a finite abelian extension $K/k$, there exist $n\in{\ma N}
\cup\{0\}$, $m\in{\ma N}$ and $N\in R_T$ such that $K\subseteq
L_n\cicl N{}k_m={_n\cicl N{}_m}$ (see \cite{Hay74} or
\cite[Theorem 12.8.31]{Vil2006}).
We define $M:=L_nk_m$. In $M/k$ no finite prime of $k$ is
ramified.

For any extension $E/F$ of global fields and for any place $\pL$
of $E$ and $\pK=\pL\cap F$, the ramification index is denoted by
$e_{E/F}(\pL|\pK)=e(\pL|\pK)$ and the inertia degree is denoted
by $f_{E/F}(\pL|\pK)=f(\pL|\pK)$. When the extension is Galois
we denote $e_{\pK}(E|F)=e_{E/F}(\pL|\pK)$ and $f_{\pK}(E|F)
=f_{E/F}(\pL|\pK)$. In particular for any abelian extension $E/k$,
$e_P(E|k)$ and $f_P(E|k)$ denote the ramification index and the inertia
degree of $P\in R_T^+$ in $E/k$ respectively, and we denote by
$e_{\infty}(E|k)$ and $f_{\infty}(E|k)$ the ramification index and
the inertia degree of $\p$ in $E/k$. 

For any finite separable extension $K/k$ the {\em finite primes}
of $K$ are the primes over the primes $P$ in $R_T^+$
and the {\em infinite primes} of $K$ are the primes over $\p$.
The {\em Hilbert class field} $H_K$ of $K$ is the maximal 
abelian extension of $K$
unramified at every finite prime of $K$ and where all the infinite
primes of $K$ are fully decomposed. The {\em genus field} $\g K$
of $K/k$ is the maximal extension of $K$ contained in $H_K$
and such that it is the composite $\g K=KF$ where $F/k$ is abelian.
We choose $F$ the maximal possible extension. In other words, $F$ is
the maximal abelian extension of $k$ contained in $H_K$.

Let $K/k$ be a finite abelian extension. We know that
$\g K=K\g {E^H}$ is the genus field of $K$ where $H$ is the
decomposition group of the infinite primes in
$KE/K$ and $E:=KM\cap \cicl N{}$ (see \cite{BaMoReRzVi2018}).
We also know that $K\g E/\g K$ and $KE/K$ are extensions
of constants.

For a local field $F$ with prime $\pK$, we denote by $F(\pK)
\cong \F$ the residue field of $F$ and by $U_{\pK}^{(n)}=1+\pK^n$
the group of $n$--th units of $F$, $n\in{\ma N}\cup\{0\}$.
Let $\pi=\pi_F=\pi_{\pK}$ be a uniformizer for $\pK$, that is,
$v_{\pK}(\pi)=1$. Then the multiplicative group of $F$ 
satisfies $\*F\cong\langle\pi\rangle\times U_{\pK}
\cong \langle\pi\rangle\times \*\F\times U_{\pK}^{(1)}$ as groups.

For a global field $F$, we denote by $J_F$ the id\'ele group of $F$
and by $C_F=J_F/F^*$ the id\'ele class group of $F$.

\section{Extended genus field of a global function field}\label{S3}

First we establish the definition of {\em extended genus fields} according
to Angl\`es and Jaulent \cite{AnJa2000}. We begin recalling two results
on class field theory.

\begin{proposition}\label{P3.0}
Let $F$ be a global (resp. local) field and let $R/F$ be the field corresponding
to the subgrup ${\mc H}<C_F=J_F/F^*$ 
(resp. ${\mc H}<F^*$), that is, ${\mc H}$
is the open subgroup of $C_F$ (resp. open subgroup of $F^*$)
such that ${\mc H}=\N_{R/F}C_R$ (resp. ${\mc H}=\N_{R/F}R^*$) and
$\Gal(R/F)\cong C_F/{\mc H}$ (resp. $\Gal(R/F)\cong R^*/{\mc H}$).
Let $E/F$ be a finite separable extension. Then the extension $ER/E$
corresponds to the group $\N^{-1}_{E/F}({\mc H})$ of $C_E$ (resp. of $E^*$).
\begin{gather*}
\xymatrix{
E\ar@{-}[rr]^{\N^{-1}_{E/F}({\mc H})}\ar@{-}[d]&&ER\ar@{-}[d]\\
F\ar@{-}[rr]_{\mc H}&&R
}
\end{gather*}
\end{proposition}

\begin{proof}
See \cite[Ch. X, Theorem 6]{Lan94} or \cite[Theorem 3.8]{RaRzVi2019}.
\end{proof}

\begin{proposition}\label{P3.-1}
Let $L/K$ be a finite separable extension of global fields. Let ${\mc H}$
be an open subgroup of finite index in $C_L$ and let $L_{\mc H}$ be
its class field. Let $K_0$ be the maximal abelian extension of K contained
in $L_{\mc H}$. Then the norm group of $K_0$ is $N_{L/K}({\mc H})$.
\end{proposition}

\begin{proof}
See \cite[Lemma 1]{Fu67} or \cite[Proposici\'on 17.6.48]{RzeVil2017}.
\end{proof}
\begin{gather*}
\xymatrix{
&&L_{\mc H}\ar@{<~>}[r]\ar@{-}[ld]\ar@{-}[dd]&{\mc H}\subseteq C_L\\
\N_{L/K}({\mc H})\ar@{<~>}[r]&K_0\ar@{-}[d]\\ & K\ar@{-}[r]&L
}
\end{gather*}

\begin{definition}\label{D3.1} Let $K$ be a global function field
and let $S$ be a finite nonempty set of prime divisors of $K$. The
{\em Hilbert class field} of $K$ relative to $S$, $H_{K,S}$, is the
maximal abelian unramified extension of $K$ such that every element
of $S$ decomposes fully in $H_{K,S}/K$.
\end{definition}

In the case of a number field $K$, the extended Hilbert class field
$H_K^+$ is, by
definition, the maximal abelian extension of $K$ such that every
finite prime of $K$ is unramified in $H_K^+$. We have that $H_K^+$
corresponds to the id\`ele group $J_K^{1_+}$, that is, a place
$\pK$ of $K$ decomposes fully in $H_K^+/K$ if and only if
$\pK$ is principal generated by a totally positive element.
So, a concept of ``{\em totally positive}'' should be developed
in the function field case.

Let $k=\F(T)$ be a fixed global rational function field over the field of $q$
elements $\F$. Let $\p$ be the infinite prime of $k$, that is, the pole of
$T$ and let $k_{\infty}\cong \F\big(\big(\frac{1}{T}\big)\big)$ be the
completion of $k$ at $\p$. Let $x\in \*{k_{\infty}}$. Then $x$ is
written uniquely as
\[
x=\Big(\frac 1T\Big)^{n_x} \lambda_x\varepsilon_x
\quad \text{with} \quad n_x\in
{\ma Z},\quad \lambda_x\in \*{\F}\quad\text{and}\quad \varepsilon_x
\in U_{\infty}^{(1)},
\]
where $U_{\infty}^{(1)}=U_{\p}^{(1)}$. We write $\pi_{\infty}:=1/T$, 
which is a uniformizer at $\p$.

\begin{definition}\label{D3.2} The {\em sign function} is defined
as $\phi_{\infty}\colon \*{k_{\infty}}\lra \*{\F}$ given by $\phi_{\infty}
(x)=\lambda_x$ for $x\in \*{k_{\infty}}$. The value $\phi_{\infty}(x)$
is called the ``{\em sign}'' of $x$. We also write $\sgn(x)=\phi_{
\infty}(x)$.
\end{definition}

We have that $\phi_{\infty}$ is an epimorphism and $\ker \phi_{\infty}=
\langle\pi_{\infty}\rangle\times U_{\infty}^{(1)}$. For example, if $M\in
R_T=\F[T]$, $M\neq 0$, say $M$ is of degree $d$ and its leader coefficient
is $a_d\in \*{\F}$, then $\sgn(M)=a_d$.

\begin{definition}\label{D3.3} 
Let $L$ be a finite separable extension of $k_{\infty}$. We define
the {\em sign of $\*L$} by the morphism $\phi_L:=\phi_{\infty}\circ
\N_{L/k_{\infty}}\colon \*L\lra \*{\F}$.
\end{definition}

\begin{proposition}\label{P3.4}
Let $k_{\infty}\subseteq E\subseteq L$. Then we have
$\phi_L=\phi_E\circ \N_{L/E}$ and $\N_{L/E}(\ker \phi_L)
\subseteq \ker\phi_E$.
\end{proposition}

\begin{proof} \cite[Remarque 1.1.1]{AnJa2000}.
\end{proof}

\begin{definition}\label{D3.5}
Let $L/k$ be a finite separable extension and let $\P$ be the set
of infinite places of $L$, that is $\P$ is the set of primes $\pK$
of $L$ such that $\pK|\p$.
Let $L_v$  the completion of $L$ at $v\in\P$. We have
that $\phi_{L_v}(x)$ is well defined for $x\in\*L$. The element $x\in\*L$ is
called {\em totally positive} if $\phi_{L_v}(x)=1$ for all $v\in\P$.
We define $L^+=\{x\in \*L\mid \text{$x$ is totally positive}\}$.
\end{definition}

\begin{definition}\label{D3.6}
The {\em group of signs} of a global function field is defined by $\Sig_L:=
\*L/L^+$. For a subgroup $R$ of $\*L$ we define $\Sig_L(R):=R L^+/L^+$,
so $\Sig_L(\*L)=\Sig_L$.
\end{definition}

\begin{proposition}\label{P3.7}
We have $\Sig_L(\*L)\cong\prod_{v\in\P}\phi_{L_v}(\*{L_v})$.
\end{proposition}

\begin{proof}
\cite[Lemme 1.2.1]{AnJa2000}.
\end{proof}

Let ${\mc O}_L=\{x\in L\mid v_{\pK}(x)\geq 0\text{\ for all $\pK\notin \P$}\}$.
Let $P_L^+=\{x{\mc O}_L\mid x\in L^+\}$ be the principal ideals
generated by a totally positive element and we define the {\em extended
ideal class group} by
\[
Cl_L^{\ext}=Cl_L^+=I_L/P_L^+,
\]
where $I_L$ is the group of fractional ideals of ${\mc O}_L$. We have 
the ideal class group $Cl_L=
I_L/P_L$ where $P_L=\{(x)=x{\mc O}_L\mid x\in \*L\}$ and
$|Cl_L|<\infty$.

\begin{definition}\label{D3.8}
We define the following subgroups of the group of id\`eles $J_L$ as:
\begin{gather*}
U_L:=\prod_{v\in\P}\*{L_v}\times \prod_{v\notin \P} U_{L_v},\\
U_L^+:=\prod_{v\in\P}\ker \phi_{L_v}\times \prod_{v\notin \P} U_{L_v}.
\end{gather*}
\end{definition}

The groups $U_L L^*$ and $U_L^+ L^*$ are open subgroups of 
$J_L$, the id\`ele group
of $L$. The groups $U_L L^*$ and $U_L^+ L^*$ correspond to 
$J_L^1$ and $J_L^{1_+}$, the id\`ele
groups congruent to the modulus $1$ and $1^+$
respectively in the number field case.

From the canonical isomorphism $I_L\cong J_L/U_L$ we obtain
\begin{gather*}
Cl_L\cong J_L/U_L\*L
\intertext{and since we may find principal id\`eles with arbitrary signs
at the infinite places, it follows that}
Cl_L^+\cong J_L/U_L^+\*L.
\end{gather*}

For $S=\P$ we denote $H_L=H_{L,S}$ the Hilbert Class Field of the
the global function field $L$. By class field theory we have
\[
\Gal(H_L/L)\cong Cl_L\cong J_L/U_L\*L.
\]

\begin{definition}\label{D3.9}
Let $H_L^+=H_L^{\ext}$ be the abelian extension of the global function
field $L$ corresponding to the id\`ele subgroup $U_L^+ \*L$ of $J_L$. The field
$H_L^+$ is called {\em the extended Hilbert Class Field} of $L$
corresponding to ${\mc O}_L$.
\end{definition}

We have that $H_L^+/L$ is an unramified extension at the finite prime
divisors of $L$, $H_L\subseteq H_L^+$ and
\[
\Gal(H_L^+/L)\cong J_L/U_L^+\*L\cong I_L/P_L^+=
Cl_L^+= Cl_L^{\ext}.
\]

Now, we have
\[
\frac {U_L}{U_L^+}\cong \prod_{v\in\P}\frac{\*{L_v}}{\ker\phi_{L_v}}
\cong \prod_{v\in\P}\phi_{L_v}(\*{L_v})\cong \frac{\*L}{L^+}.
\]

\begin{remark}\label{R3.10}{\rm{
We have $[H_L^+:H_L]||\Sig(\*L)|=\prod_{v\in\P}|\phi_{L_v}(\*{L_v})|
|(q-1)^{|S|}$.
}}
\end{remark}

It follows that $H_L^+/H_L$ is unramified at the finite primes and 
tamely ramified at the infinite primes. Furthermore
\[
\Gal(H_L^+/H_L)\cong \ker\Big(\Gal(H_L^+/L)
\xrightarrow{\rest}{}\Gal(H_L/L)\Big)\cong
\frac{U_L\*L}{U_L^+\*L}\cong \frac{\Sig_L}{\Big(\frac{U_L\cap U_L^+
\*L}{U_L^+}\Big)}.
\]

\begin{proposition}\label{P3.11}
Let $L/E$ be a finite separable extension of global function fields. Then
$H_E^+\subseteq H_L^+$. If $L/E$ is a Galois extension, then
$H_L^+/E$ is also Galois. 
\end{proposition}

\begin{proof}
We have
\[
\xymatrix{
& H_E^+\ar@{-}[r]\ar@{-}[d] & LH_E^+\ar@{-}[d]\\
& H_E^+\cap L\ar@{-}[r]\ar@{-}[dl]&L\\ E}
\]
Then $H_E^+/H_E^+\cap L$ is an abelian extension unramified
at the finite primes so that $LH_E^+/L$ is also unramified at the
finite primes.

We consider the norm $\N_{L/E}\colon C_L=J_L/\*L\lra C_E=J_E/\*E$
of id\`ele's groups and since $\N_{L_w/E_v}(\ker \phi_{L_w})\subseteq
\ker \phi_{E_v}$, $\N_{L/E}(U_L^+)\subseteq U_E^+$ and $\N_{L/E}(
\*L)\subseteq \*E$, it follows $\N_{L/E}(U_L^+\*L)\subseteq U_E^+\*E$.
Now the group $U_E^+\*E$ corresponds to the field $H_E^+$.
From Proposition \ref{P3.0}, we have that 
the norm group of $LH_E^+$ is $N_{L/K}^{-1}(U_E^+\*E)\supseteq
U_L^+\*L$. Hence $LH_E^+\subseteq H_L^+$ and thus
$H_E^+\subseteq H_L^+$.

Now, if $L/E$ is a Galois extension, we consider $\sigma\colon
H_L^+\lra \overline{H_L^+}$ a monomorphism such that $\sigma|_E
=\Id_E$. Here $\overline{H_L^+}$ denotes a fixed algebraic closure
of $H_L^+$. Since $L/E$ is Galois, we have $\sigma|_L\colon L\lra \bar{L}$
and since $\sigma|_E=\Id_E$, it follows that $\sigma(L)=L$. On the
other hand, since $\sigma(\ker \phi_{L_w})=\ker \phi_{L_{\sigma(w)}}$
we have $\sigma(U_L^+)=U_L^+$ and $\sigma(\*L)=\*L$ so that
$\sigma(U_L^+\*L)=U_L^+\*L$. Therefore $\sigma(H_L^+)=H_L^+$,
that is, $\sigma\in\Aut_E(H_L^+)$ and hence $H_L^+/E$ is a
Galois extension.
\end{proof}

A general property of the HCF (resp. of the extended HCF) is given
in the following result.

\begin{proposition}\label{P3.14*}
Let $L/k$ be a finite separable extension of global function fields.
Let ${\mc O}_L$ be the integral closure of $\F[T]$ in $L$. Then a
fractional ideal $\pK$ of ${\mc O}_L$ decomposes fully in
\las
\item $H_L/L$ if and only if $\pK$ is principal;

\item $H_L^+/L$ if and only if $\pK$ is principal generated
by an element of $L^+=\{x\in L\mid \phi_{L_{\pK}}(x)=1 \text{\ 
for all\ }\pK|\p\}$.
\end{list}
\end{proposition}

\begin{proof}
From class field theory, we
have that $\pK$ decomposes fully in $H_L/L$
(resp. in $H_L^+/L$) if and only if $\theta\lceil
\*{L_{\pK}}\rceil_{\pK}:=\{(\ldots, 1, x, 1,\ldots)\mid x\in\*{L_{\pK}}\}
\subseteq \*L U_L$ (resp. $\subseteq \*L U_L^+$)
(see \cite[III, Theorem 8.3]{Neu69} or
\cite[Corolario 17.6.47]{RzeVil2017}) if and only if for all $x\in\*{L_{\pK}}$
there exist $\beta_x\in\*L$ and $\vec\alpha_x\in U_L$ (resp.
$\vec\alpha_x\in U_L^+$) such that
$\theta\lceil x\rceil_{\pK}=\beta_x\vec\alpha_x$.
Thus $\vec\alpha_x=(\ldots, \beta_x^{-1},\ldots, \beta_x^{-1},\beta_x^{-1}x,
\beta_x^{-1},\ldots,\beta_x^{-1},\ldots)$. This is equivalent to $v_{\mathfrak q}
(\beta_x)=0$ for all ${\mathfrak q}\neq \pK$ and ${\mathfrak q}\nmid\p$
and $\big(\beta_x^{-1}\big)_{{\mathfrak q}|\p}\in \prod_{{\mathfrak
q}|\p}\*{L_{\mathfrak q}}$ (resp. $\big(\beta_x^{-1}\big)_{{\mathfrak 
q}|\p}\in \prod_{{\mathfrak q}|\p}\ker\phi_{L_{\mathfrak q}}$).

Let $x\in\*{L_{\pK}}$. The principal ideal in ${\mc O}_L$ generated by
$\beta_x$ satisfies $\langle \beta_x
\rangle=\pK^{n_x}=\beta_x{\mc O}_L$. In particular for $x\in\*{L_{\pK}}$ with
$v_{\pK}(\beta_x)=1$ we have $\langle \beta_x\rangle=\pK$, $\beta_x\in
\*L$. Hence $\pK$ is a principal ideal $\pK=\langle \beta_x\rangle$ (resp.
$\pK$ is a principal ideal $\pK=\langle \beta_x\rangle$
and $\beta_x\in\ker\phi_{L_{\mathfrak q}}$ for all ${\mathfrak q}|\p$).
\end{proof}

\begin{definition}\label{D3.15}
Let $L/K$ be a finite separable extension of global function fields. We
define the {\em extended genus field} $L_{gex,K}$ of $L$
with respect to $K$ as the maximal extension of $L$ 
contained in $H_L^+$ that is of the form $LK_1^+$ where $K_1^+/
K$ is an abelian extension. The maximal field $K_1^+$ that
satisfies $L_{gex,K}=LK_1^+$ will be denoted by $H_{L/K}^+$
In this way, $H_{L/K}^+$ is the maximal abelian extension of $K$ such
that $L_{gex, K}=LH_{L/K}^+$.
\end{definition}

\[
\xymatrix{
L\ar@{-}[r]\ar@{-}[d]&LH_{L/K}^+=L_{gex, K}\ar@{-}[r]\ar@{-}[d]
&H_L^+\\
K\ar@{-}[r]&H_{L/K}^+
}
\]

\begin{remark}\label{R3.16}{\rm{
When $L/K$ is abelian, $L_{gex,K}$ is the maximal abelian
extension of $K$ contained in $H_L^+$.

We also have $H_{L/K}^+=\ge{(H_{L/K}^+)}$.
}}
\end{remark}

When $K=k=\F(T)$ and $L/k$ is a finite separable extension
we have $\ge L=L_{{\eu{gex}},k}$.

\begin{remark}\label{R3.11-1*}{\rm{
We have that if $L/E$ is a finite separable extension of global function fields, then
$\g E\subseteq \g L$, $\ge E\subseteq \ge L$, $H_E\subseteq H_L$
and $H_E^+\subseteq H_L^+$.
We also have $\ge K=\ge{(\ge K)}$.
}}
\end{remark}

\begin{remark}\label{R3.17}{\rm{
In \cite{RaRzVi2019} we defined the {\em extended genus field}
$K^{\ext}$ for
a finite abelian extension $K/k$ as $K^{\ext}:=\ge E K= E^{\ext} K$ 
where $E=KM\cap \cicl N{}$
is the corresponding cyclotomic function field and $\ge E
=E^{\ext}$ is defined
as the maximal cyclotomic extension containing $E$ such
that the finite primes are unramified (we prove in 
Theorem \ref{T4.4} that this agrees
with Definition \ref{D3.15} in case $K$ is contained in
a cyclotomic function field). 
}}
\end{remark}

\section{Extended genus fields in the cyclotomic case}\label{S4}

Let $E\subseteq \cicl N{}$ be a cyclotomic function field.
We have $\Gal(\cicl N{}/k)\cong \big(R_T/\langle N\rangle\big)^*$.
Let $X$
be the group of Dirichlet characters associated to $E$, that is,
\begin{align*}
X=\{\chi\colon (R_T/\langle N\rangle)^{\ast}\lra \*{\ma C}\mid &
\text{$\chi$ is a group homomorphism}\\
&\text{such that $\Gal(\cicl N{}/E)
\subseteq \ker\chi$}\}
\end{align*}
and $E=\cicl N{}^H$ is the fixed field under 
$H$, where $H=\bigcap_{\chi\in X}\ker \chi$. We have
$X\cong \widehat{\Gal(E/k)}=\Hom(\Gal(E/k), \*{{\ma C}})$
the group of characters of $\Gal(E/k)$. The maximal
cyclotomic extension of $E$ unramified at the finite primes is
the field $E^{\ext}$ associated to $Y:=\prod_{P\in R_T^+} X_P$ where
$X_P=\{\chi_P\mid \chi\in X\}$ and $\chi_P$ denotes the
$P$--th component of $\chi$ (see \cite{MaRzVi2013}).

The aim of this section is to prove that $E^{\ext}=\ge E$. In fact, in
\cite{RaRzVi2019} the extended genus field of $E/k$ was defined as
$E^{\ext}$ and it was denoted as $\ge E$ in that paper.

To begin with, we first find the id\`ele class group associated to
a cyclotomic function field $\cicl N{}$. Let $N=P_1^{\alpha_1}\cdots
P_r^{\alpha_r}$ be the decomposition of $N$ as a product of
irreducible monic polynomials.  Set $R_T'=R_T
\setminus\{P_1,\ldots,P_r\}$, $\pi=1/T=\pi_{\infty}$
and $U_{\infty}=U_{\p}$.

We define
\begin{gather}\label{ideleskN}
{\mc X}_N=\prod_{i=1}^r U_{P_i}^{(\alpha_i)}\times \prod_{P\in R_T'} U_P
\times [\langle\pi\rangle\times U_{\infty}^{(1)}].
\end{gather}

Note that ${\mc X}_N\*k/\*k\cong {\mc X}_N$.

Let $U_T:=\{\vec\alpha\in J_k\mid \alpha_{\p}=1 \text{\ and
$\alpha_P\in U_P$ for all $P\in R_T^+$}\}$.
D. Hayes \cite{Hay74} proved that $U_T\cong G_T=\Gal(
k_T/k)$ where $k_T:=\bigcup_{N\in R_T}\cicl N{}$.

\begin{proposition}\label{P4.1}
Let $U':=\prod_{P\in R_T}U_P\times 
[\langle\pi\rangle\times U_{\infty}^{(1)}] \subseteq J_k$.
Then there exists an epimorphism
$\psi_N\colon U'\lra
\Gal(\cicl N{}/k)=:G_N$ with $\ker \psi_N={\mc X}_N$
so that, $U'/{\mc X}_N\cong G_N$.
\end{proposition}

\begin{proof}
Let $\vec\xi\in U'$. Then 
$\xi_{P_i}\in U_{P_i}=\{\sum_{j=0}^{\infty} a_jP_i^j\mid a_j\in R_T/
\langle P_i\rangle, a_0\neq 0\}$,
$1\leq i\leq r$. Since $k\subseteq k_{P_i}$ 
is dense, there exists $Q_i\in R_T$ with
$Q_i\equiv \xi_{P_i}\bmod P_i^{\alpha_i}$. From the Chinese Residue
Theorem, there exists $C\in R_T$ such that 
$C\equiv Q_i\bmod P_i^{\alpha_i}$, $1\leq i\leq r$
and therefore $C\equiv \xi_{P_i}\bmod P_i^{\alpha_i}$, $1\leq i\leq r$.

Now, if $C_1\in R_T$ satisfies $C_1\equiv \xi_{P_i}
\bmod P_i^{\alpha_i}$, $1\leq i\leq r$, then $P_i^{\alpha_i}
|C-C_1$ for $1\leq i\leq r$. It follows that $N|C-C_1$ so
that $C\in R_T$ is unique modulo $N$. On the other hand,
$v_{P_i}(\xi_{P_i})=0$, so that $P_i\nmid \xi_{P_i}$.
It follows that $\gcd(C,N)=1$. In this way we obtain that
$C\bmod N$ defines an element of $G_N$.

Given $\sigma\in G_N$, there exists $C\in R_T$ such that $\sigma
\lambda_N=\lambda_N^C$ where $\lambda_N$
is a generator of $\Lambda_N$. Let $\vec\xi\in U'$
with $\xi_{P_i}=C$, $1\leq i\leq r$ and $\xi_P=1=\xi_{\infty}$
for all $P\in R_T'$. Hence $\vec\xi\mapsto C\bmod
N$ and $\psi_N$ is surjective. Finally, $\ker\psi_N
=\{\vec\xi\in U'\mid\xi_{P_i}\equiv 1 \bmod P_i^{\alpha_i},
1\leq i\leq r\}={\mc X}_N$. The result follows.
\end{proof}

Next, we will show that $U'/{\mc X}_N\cong J_k/{\mc X}_N\*k$.
We have the natural composition
\[
\xymatrix{
U'\ar@{^{(}->}[r]\ar@/_1pc/@{>}_{\mu}[rr]&
J_k\ar@{>>}[r]&J_k/{\mc X}_N \*k,
}
\]
with $\im \mu=U'{\mc X}_N\*k/{\mc X}_N\*k$ and $\ker \mu =
U'\cap {\mc X}_N\*k$.

Now, ${\mc X}_N\subseteq U'$ which implies ${\mc X}_N
\subseteq U'\cap {\mc X}_N\*k$. Conversely, if $\vec\xi
\in U'\cap {\mc X}_N\*k$, then the components of
$\vec\xi$ are given by
\begin{align*}
\xi_P&=a\cdot \beta_P, \quad P\in R_T, \quad
\vec\beta \in {\mc X}_N, \quad a\in \*k,\\
\xi_{\infty}&=a\cdot\beta_{\infty}, \quad
\beta_{\infty}\in \langle\pi\rangle\times U_{\infty}^{(1)}.
\end{align*}

Since $\xi_P,\beta_P\in U_P$ we have that $v_P(\xi_P)=
v_P(\beta_P)=0$ for all $P\in R_T$. It follows that 
$v_P(a)=0$ for all $P\in R_T$. Furthermore, since 
$\deg a=0$, we have $v_{\infty}(a)=0$ so that $a\in\*\F$.

Next, since $\xi_{\infty}, \beta_{\infty}
\in\langle\pi\rangle\times U_{\infty}^{(1)}=\ker \phi_{\infty}$,
it follows that $1=\phi_{\infty}(\xi_{\infty})=
\phi_{\infty} (a) \phi_{\infty} (\beta_{\infty})=
\phi_{\infty} (a)$. Thus
$a=1$. It follows that $\vec\xi\in{\mc X}_N$.
Therefore $\ker \mu={\mc X}_N$ and we obtain a monomorphism
$U'/{\mc X}_N\stackrel{\theta}
{\hooklongrightarrow} J_k/{\mc X}_N \*k$.

It remains to verify that $\theta$ is onto. We must prove
that $J_k=U'{\mc X}_N \*k =U'\*k$. 
We have that $U'$ corresponds to an abelian
extension of $k$ unramified at every finite prime. Let
$L/k$ be this extension. Now $U_{\infty}^{
(1)}$ corresponds to the first ramification group and
therefore corresponds to the wild ramification of $\p$.
It follows that in $L/k$ there is at most a unique 
ramified prime and this prime is tamely ramified and
of degree $1$ ($\p$). It follows that $L/k$
is an extension of constants 
(see \cite[Proposici\'on 10.4.11]{RzeVil2017}).

Finally, since $d=\min\{n\in{\ma N}\mid \deg\vec
\alpha=n, \vec\alpha\in U'\}=1$, the field of constants of
$L$ is $\F$ and therefore $L=k$. It follows that 
$\N_{L/k} C_L=C_k\cong \*kU'/\*k
\cong U'$, that is, $J_k/\*k\cong U'$. Hence
$J_k=\*k U'$.

We have proved

\begin{proposition}\label{P4.2}
Let $N=P_1^{\alpha_1}\cdots P_r^{\alpha_r}\in R_T$. 
Then the id\`ele subgroup ${\mc X}_N\*k$ of $J_k$
that corresponds to the cyclotomic function field
$\cicl N{}$ is the subgroup
\begin{gather*}
{\mc X}_N=\prod_{i=1}^r U_{P_i}^{(\alpha_i)}\times \prod_{P\in R_T'} U_P
\times [\langle\pi\rangle\times U_{\infty}^{(1)}]. 
\intertext{and consequently, the id\`ele class subgroup of $C_k$ corresponding
to $\cicl N{}$ is}
{\mc X}_N\*k/\*k\cong {\mc X}_N. \tag*{$\fin$}
\end{gather*}
\end{proposition}

\begin{corollary}\label{C4.3}
The id\`ele subgroup corresponding to the real
cyclotomic function field $\cicl N{}^+$ is ${\mc X}_N^+\*k$ where
\[
{\mc X}_N^+=\prod_{i=1}^r U_{P_i}^{(\alpha_i)}\times \prod_{P\in R_T'} U_P
\times \*{k_{\infty}}.
\]
\end{corollary}

\begin{proof}
If $L_1$ is the field associated to ${\mc X}_N^+\*k$, 
then following the previous argument we obtain that
the filed of constants of $L_1$ is $\F$ and we have
\[
\frac{{\mc X}_N^+}{{\mc X}_N}\cong\frac{\*{k_{\infty}}}{\langle\pi\rangle\times
U_{\infty}^{(1)}}\cong \*\F,
\]
which implies that $[L:L_1]\leq q-1$. Furthermore $\p$ decomposes
totally in $L_1$ and it has ramification index $q-1$ in $L$.
It follows that $[L:L_1]= q-1$ and $L_1=\cicl N{}^+$. In fact,
${\mc X}_N^+\cap \*k =\*\F$.
\end{proof}

We return to our study of $\ge E$ for $E\subseteq\cicl N{}$. First we give
a result needed for the main result.

\begin{proposition}\label{P5.0}
Let $E$ be a cyclotomic function field, $E\subseteq \cicl N{}$. If $\pK$ is an
infinite prime in $E$ then $1/T$ is a norm from $E_{\pK}$, that
is, there exists $x\in E_{\pK}$ such that $\N_{E_{\pK}/k_{\infty}} x
=1/T$.
\end{proposition}

\begin{proof}
Since $E$ is cyclotomic, we have that $E\subseteq k_{\infty}
(\sqrt[q-1]{-1/T})$ (\cite[\S 3]{BaKo96} or
\cite[Proposici\'on 9.3.20]{RzeVil2017}). Hence
$E_{\pK}=k_{\infty}(\sqrt[e]{-1/T})$ with $e|q-1$. The result
follows. See also \cite[Th\'eor\`eme 1.1.2 (ii)]{AnJa2000}.
\end{proof}

Let $\Gamma\subseteq J_E$ be the id\`ele subgroup of $E$
that corresponds to $H_E^+$. This group satisfies 
$\N_{H_E^+/E}C_{H_E^+}=\*E \Gamma/\*E$. Then, by definition,
we may take $\Gamma=U_E^+$. Observe also that since $E/k$ is
abelian, $H_{E/k}^+$ corresponds to $\ge E$ over $k$ (see 
\cite[Proposition 2.1.3]{AnJa2000} or \cite[Proposici\'on 
17.6.48]{RzeVil2017}).
It follows that the id\`ele class subgroup of
$C_k$ that corresponds to $H_{E/k}^+=\ge E$ over $k$ is
precisely $\N_{E/k}(U_E^+)\*k/\*k\cong \N_{E/k} U_E^+$.

We have $U_E^+=\prod_{\pK|\infty}\ker 
\phi_{E_{\pK}}\times \prod_{\pK\nmid 
\infty} U_{\pK}$, where we denote $\infty=\p$.

If $\pK$ is unramified, that is, $\pK\nmid \infty$ and $\pK\nmid
P_i$, $1\leq i\leq r$, where $N=P_1^{\alpha_1}\cdots P_r^{\alpha_r}$, then
we have $\N_{E_{\pK}/k_P}(U_{\pK})=U_P$, where $\pK\cap k=P\in R_T^+$
(\cite[II, Corollary 4.4]{Neu69} or
\cite[Teorema 17.2.17]{RzeVil2017}). If $\pK|P_i$ for some
$1\leq i\leq r$, we have
$[U_{P_i}:\N_{E_{\pK}/k_{P_i}} U_{\pK}]=e_{E_{\pK}/k_{P_i}}(\pK|P_i)
=\Phi(P_i^{\alpha_i})=q^{(\alpha_i-1)d_i}(q-1)$ where $d_i=
\deg P_i$.

Now, because $E\subseteq \cicl N{}$, we obtain that $\g E\subseteq
\cicl N{}$ and the subgroup of $J_E$ that corresponds to
$E_H$ is $U_E\*E$ and therefore the id\`ele class subgroup 
of the id\`ele class group $C_k$
corresponding to $\g E$ is $\N_{E/k}(U_E)\*k/\*k$. We have
\begin{gather*}
\N_{E/k} U_E=
\prod_{\pK|\p}\N_{E_{\pK}/k_{\infty}} \*{E_{\pK}}\times \prod_{P\in R_T}
\prod_{\pK|P}\N_{E_{\pK}/k_P} U_{\pK}.
\intertext{We also have}
N_{E/k} U_E^+=
\prod_{\pK|\p}\N_{E_{\pK}/k_{\infty}}(\ker \phi_{E_{\pK}})\times \prod_{P\in R_T}
\prod_{\pK|P}\N_{E_{\pK}/k_P} U_{\pK}.
\end{gather*}

If we write $R_T':=R_T\setminus
\{P_1,\ldots,P_r\}$, then
\begin{gather}\label{Ec4.2}
\Big(\prod_{P\in R_T'}U_P\times \prod_{i=1}^r U_{P_i}^{(\alpha_i)}\Big)\*k\subseteq
\Big(\prod_{P\in R_T}\prod_{\pK|P}\N_{E_{\pK}/k_P} U_{\pK}\Big)\*k.
\end{gather}

Let $\pK$ be an infinite prime in $E$.
Note that since $E$ is cyclotomic,
we have $E_{\pK}\subseteq k_{\infty}(\sqrt[q-1]{
-1/T})$ and it follows that $1/T$ is 
a norm from $E_{\pK}$ (Proposition \ref{P5.0}. See also
\cite[Th\'eor\`eme 1.1.2]{AnJa2000}).
Because $E_{\pK}/k_{\infty}$ is totally ramified, we have
$\N_{E_{\pK}/k_{\infty}} \*{E_{\pK}}\supseteq \langle\pi_{\infty}\rangle\times
U_{\infty}^{(n_0)}$ for some $n_0\in{\ma N}\cup \{0\}$ 
(\cite[II, Theorem 7.17]{Neu69} or \cite[Teorema 17.5.47]{RzeVil2017}). 

Let $\*{E_{\pK}}=\langle\pi_{E_{\pK}}
\rangle\times \*\F\times U_{\pK}^{(1)}$ and
$\phi_{\pK}:=\phi_{E_{\pK}}=\phi_{\infty}\circ \N_{E_{\pK}/k_{\infty}}$.
We have $\ker\phi_{\pK}=\N_{E_{\pK}/k_{\infty}}^{-1}(\ker\phi_{\infty})=
\N_{E_{\pK}/k_{\infty}}^{-1}\big(\langle\pi_{\infty}\rangle\times U_{\infty}^{(1)}\big)$.

Since $k\subseteq E\subseteq \cicl N{}$, it follows that $k_{\infty}\subseteq E_{\pK}
\subseteq \cicl N{}_{\pL}=k_{\infty}(\sqrt[q-1]{-1/T})$ where $\pL$ is a prime
in $\cicl N{}$ such that $\pL|\infty$ and $\pK=\pL\cap E$. The field associated to
$\ker\phi_{\infty}$ is $k_{\infty}(\sqrt[q-1]{-1/T})$ and the field associated to
$\ker\phi_{\pK}=\N^{-1}_{E_{\pK}/k_{\infty}}(\ker\phi_{\infty})$ is $E_{\pK}
(\sqrt[q-1]{-1/T})=k_{\infty}(\sqrt[q-1]{-1/T})$ (Proposition \ref{P3.0} or 
\cite[Corollaire 1.1.3]{AnJa2000}). It follows that $\N_{E_{\pK}/k_{\infty}}
(\ker \phi_{\pK})=\ker \phi_{\infty}$. Therefore
\begin{align*}
\N_{E/k}U_E^+&=\prod_{P\in R_T}\prod_{\pK|P}\N_{E_{\pK}/k_{\infty}}U_{\pK}\times
\prod_{\pK|\infty}\N_{E_{\pK}/k_{\infty}}(\ker \phi_{\pK})\\
&=\prod_{P\in R_T}\prod_{\pK|P}\N_{E_{\pK}/k_{\infty}}U_{\pK}\times
(\ker \phi_{\infty}).
\end{align*}

From (\ref{Ec4.2}) we obtain that
\begin{multline*}
\Big(\prod_{P\in R_T'}U_P\times \prod_{i=1}^r U_{P_i}^{(\alpha_i)}\times
(\langle\pi_{\infty}\rangle\times U_{\infty}^{(1)})\Big)\*k\\
\subseteq
\Big(\prod_{P\in R_T}\prod_{\pK|P}\N_{E_{\pK}/k_P} U_{\pK}\times
\prod_{\pK|\infty}\N_{E_{\pK}/k_{\infty}}(\ker \phi_{\pK})\Big)\*k.
\end{multline*}

Hence, ${\mc X}_N\*k\subseteq (\N_{E/k}U_E^+)\*k$.
Therefore $\ge E\subseteq \cicl N{}$.

Now, $\ge E/E$ is unramified at the finite primes and if $L$
is the field associated to $Y=\prod_{P\in R_T}X_P$, $L$
is the maximal abelian extension of $E$ contained in
$\cicl N{}$ unramified at the finite primes. Therefore
$\ge E\subseteq L$.

To show $L\subseteq \ge E$, it suffices to prove that $L\subseteq H_E^+$
since $L/k$ is abelian and $\ge E\subseteq H_E^+$ is the maximal
abelian extension of $k$ contained in $H_E^+$.

Now, to show that $L\subseteq H_E^+$, we must prove that
$\N_{L/E}C_L\supseteq \N_{H_E^+/E} C_{H_E^+}=U_E^+\*E/\*E$, where
$U_E^+=\prod_{\pK|\infty}\ker\phi_{\pK}\times \prod_{\pK\nmid\infty}U_{\pK}$.

Let $\N_{L/E}C_L=\Lambda \*E/\*E$ where $\Lambda =\N_{L/E}J_L$. 
It suffices to prove that $U_E^+\subseteq \Lambda$. 
Since $L/E$ is unramified at the finite primes, if
$\pK$ is a finite prime of $E$ and $\pL$ is a prime of $L$ over 
$\pK$, then $\N_{L_{\pL}/E_{\pK}} U_{\pL}=U_{\pK}$
where we denote $U_{\pL}=U_{L_{\pL}}$ and $U_{\pK}=U_{E_{\pK}}$
(\cite[II, Corollary 4.4]{Neu69} or 
\cite[Teorema 17.2.17]{RzeVil2017}). In particular
$\N_{L_{\pL}/E_{\pK}}\*{L_{\pL}}\supseteq U_{\pK}$ and
$\N_{L/E}J_L\supseteq \prod_{\pK\nmid \infty} U_{\pK}$.

On the other hand, $L/\g E$ is totally ramified at the infinite primes
and the infinite primes of $E$ decompose fully in $\g E$,
so that $(\g E)_{\pK}=E_{\pK}$ for $\pK|\infty$ 
and the uniformizer  of 
$\pK$ in $E$ is also a uniformizer for $\pK$ in $\g E$.

From local class fiel theory, (\cite[II, Theorem 7.17]{Neu69}
or \cite[Teorema 17.5.47]{RzeVil2017}),
$\pi_{E,\infty}:=\pi_{E_{\pK}}$ is a norm from
$L_{\pL}$ with $\pL|\pK$. If $(\underline{\ \ },L_{\pL}/E_{\pK})$
represents the Artin local map, then $(U_{E_{\pK}}^{(1)},L_{\pL}/
E_{\pK})=G^1(L_{\pL}/E_{\pK})$, the first ramification group of
$L_{\pL}/E_{\pK}$. Since the infinite primes are tamely ramified
in $L_{\pL}/E_{\pK}$, it follows that $G^1(L_{\pL}/E_{\pK})=
\{1\}$ and that
 $U_{E_{\pK}}^{(1)}\subseteq \N_{L_{\pL}/E_{\pK}}\*L_{\pL}$.

The field of constants of $E$ and of $k$ is $\F$ and every
infinite prime has degree $1$ in $E$, so that if $\pK$ is
an infinite prime of $E$, $\*{E_{\pK}}=
\langle\pi_{E,\infty}\rangle\times \*\F\times 
U_{E_{\pK}}^{(1)}$.
Note that $[E_{\pK}:k_{\infty}]=ef=e=
[\*{k_{\infty}}:\N_{E_{\pK}/k_{\infty}}\*{E_{\pK}}]$
where $e$ denotes the ramification index of $E_{\pK}/k_{\infty}$. 
Furthermore,
since $E_{\pK}/k_{\infty}$ is total and tamely ramified,
$\langle\pi_{\infty}\rangle\times U_{\infty}^{(1)}\subseteq 
\N_{E_{\pK}/k_{\infty}}(\*{E_{\pK}})$
and $\N_{E_{\pK}/k_{\infty}}\*\F=(\*\F)^e$. It follows that
$\langle\pi_{\infty}\rangle\times (\*\F)^e\times U_{\infty}^{(1)}\subseteq \N_{E_{\pK}/
k_{\infty}} \*{E_{\pK}}$. From $[\*{k_{\infty}}:\N_{E_{\pK}/k_{\infty}}\*{E_{\pK}}]=e$
we obtain the equality
\[
\N_{E_{\pK}/k_{\infty}} \*{E_{\pK}}=
\langle\pi_{\infty}\rangle\times (\*\F)^e\times U_{\infty}^{(1)}.
\]

Therefore
$\phi_{\pK}(\*{E_{\pK}})=\phi_{\infty}(\N_{E_{\pK}/k_{\infty}} \*{E_{\pK}})=
\phi_{\infty}\big(\langle\pi_{\infty}\rangle\times (\*\F)^e\times U_{\infty}^{(1)})=
(\*\F)^e$. It follows that $\ker \phi_{\pK}=\langle\pi_{E,\infty}\rangle\times R\times
U_{E_{\pK}}^{(1)}$ where $R=\{\lambda\in\*\F\mid \lambda^e=1\}=
(\*\F)^{(q-1)/e}$. Here $e$ denotes the ramification index of $\p$ in
$E/k$.

We note that if $\pL$ is an infinite prime of $L$, $\N_{L_{\pL}/E_{\pK}}
\*\F=(\*\F)^{e'}$ where $e':=e_{\infty}(L/\g E)=e_{\infty}(L/E)$.
Then $e'e=e_{\infty}(L/k)$ and in particular 
$e'e|q-1$ and $e'\big|\frac{q-1}{e}$. 
Let $\*\F=\langle\beta\rangle$ with $o(\beta)=q-1$.
Let $\lambda\in R$, $\lambda^e=1$. We have $\lambda=\beta^s$
for some $s$. Therefore $\lambda^e=\beta^{es}=1$ and $q-1|es$.
Since $e'e|q-1$ it follows that $e'e|se$ and $e'|s$. In this way
we obtain that $\lambda=\beta^s=(\beta^{s/e'})^{e'}\in
(\*\F)^{e'}\subseteq \N_{L_{\pL}/E_{\pK}} \*{L_{\pL}}$.

Hence $\ker\phi_{\pK}\subseteq \N_{L_{\pL}/E_{\pK}}\*{L_{\pL}}$
and $U_E^+=\prod_{\pK|\infty}\ker\phi_{\pK}\times \prod_{
\pK\nmid \infty}U_{\pK}\subseteq \N_{L/k}J_L$ from where we
obtain that $L\subseteq H_E^+$ and therefore $L=\ge E$.

We have proved

\begin{theorem}\label{T4.4} Let $E\subseteq \cicl N{}$. Then
the extended genus field $\ge E$ relative to $k$ is the field
associated to the group of Dirichlet characters $Y=\prod_{P\in
R_T}X_P$, where $X$ is the group of Dirichlet characters
associated to the field $E$. $\fin$
\end{theorem}

\section{Finite abelian extensions}\label{S5}

Consider $K/k$ a finite abelian extension. 
Let $n\in{\ma N}\cup\{0\}$, $m\in{\ma N}$ and
$N\in R_T$ such that $K\subseteq {_n\cicl N{}_m}$. Let 
$E:=KM\cap \cicl N{}$. We have that $EK/K$ is an extension
of constants, in particular an unramified extension (see 
\cite{BaMoReRzVi2018}). Let $H$ be the decomposition group of
the infinite primes of $K$ in
$EK/K$. We have $H$ is canonically isomorphic with the
decomposition group of the infinite primes of $\g E K/K$. We have
that $|H|$ equals the inertia degree of the infinite primes of $K$ in
$EK/K$ and in $\g EK/K$. We also have that $\g K=\g E^H K$.

\begin{theorem}\label{T5.1}
With the above notations, we have that $\ge K = DK$ with $\ge{(\g{E^H})}
\subseteq D\subseteq \ge E$.
In particular, when $H=\{1\}$, we have $\ge K=\ge E K$.
\end{theorem}

\begin{proof}
We know that $\g EM=\g KM$ (see \cite[page 2111]{BaMoReRzVi2018}).
Set $C:=\ge K M\cap \cicl N{}\supseteq \g K M\cap \cicl N{}=\g E M\cap
\cicl N{}\supseteq \g E\cap \cicl N{}=\g E$.

\[
\xymatrix{
\cicl N{}\ar@{-}[d]\\ C\ar@{-}[rr]\ar@{-}[d]&&CM=\ge K M\ar@{-}[d]\\
\g E\ar@{-}[rr]&&\g EM=\g K M
}
\]

Note that $\ge K M/\g K M$ is unramified at the finite primes because
it is so in $\ge K/\g K$. We also have that $\g E M/\g E$ is unramified
at the finite primes. In particular $C/\g E$ is unramified at the finite primes
and $C\subseteq \cicl N{}$. Since $\ge E/k$ is the maximal ciclotomic
extension with $\ge E/E$ unramified at the finite primes, it follows that
$C\subseteq \ge E$.

Therefore $\g E\subseteq C\subseteq \ge E$ and 
$\begin{array}{c}\g E M\\ \|\\ \g K M\end{array} \subseteq \begin{array}{c}
CM\\ \| \\ \ge K M\end{array}\subseteq \begin{array}{c}\ge EM\\\|\\ 
\ge E M\end{array}$. In particular $\ge K M\subseteq \ge E M$.

\[
\xymatrix{
&\cicl N{}\ar@{-}[d]\\
&\ge E\ar@{-}[rr]\ar@{-}[d]\ar@{-}[ldd]&&\ge EM\ar@{-}[d]\\
&C\ar@{-}[rr]\ar@{-}[dd]&&\ge K M\ar@{-}[dl]\\
D\ar@{-}[rr]\ar@{-}[dd]&&\ge K=DK\ar@{-}[dd]\\
&\g E\ar@{-}[dl]\ar@{-}[rr]&&\g E M=\g K M\ar@{-}[uu]\ar@{-}[dl]\\
\g {E^H}\ar@{-}[rr]&& \g K=\g{E^H}K
}
\]

Let $D:=\ge K\cap \cicl N{}=\g{E^H} K\cap \cicl N{}\supseteq \g{E^H}
\cap \cicl N{} =\g{E^H}$. Thus $\g{E^H}\subseteq D$ and $\ge K=DK$.

We have that $EM/E$ is unramified at the finite primes. Therefore $DEM=
DKM/EM=KM$ is unramified at the finite primes since $\ge K/K$ is so.
It follows that $DE/E$ is unramified at the finite primes so that $D
\subseteq \ge E$. Hence $\ge K=DK\subseteq \ge E K$.

In the particular case that $H=\{1\}$ we have that $E\subseteq \g K
\subseteq \ge K$ so that $\ge E\subseteq \ge K$ and $\ge K=\ge E K$.

In the general case, $\g{E^H}\subseteq D\subseteq \ge E$ and $\ge E/
\g{E^H}$ is totally ramified at the infinite primes. Since $\g{E^H}\subseteq
\g K$, it follows that $\ge{(\g{E^H})}\subseteq \ge K$. Hence $\ge{(\g{
E^H})}\subseteq D$.
\end{proof}

\begin{example}\label{Ex5.3}{\rm{
Let $K:=k(\sqrt[l]{\gamma D})$ where $l$ is a prime number such that
$l|q-1$, $D\in R_T$, $D$ monic with $l\nmid\deg D$, 
and $\gamma\not\equiv (-1)^{\deg D}
\bmod (\*\F)^l$. Then $E=MK\cap \cicl D{}=k(\sqrt[l]{\*D})$ where $M=k_l$
and $D^*=(-1)^{\deg D} D$. Then
$H\cong C_l$, $\ge E=\g E= E$, $\g{E^H}=E^H=k$, $\g K=K$.

Set $n:=\deg D$, $D=T^n+a_{n-1}T^{n-1}+\cdots +a_1T+a_0$, $n=lm-r$ with
$0<r<l$.
The infinite prime $\p$ ramifies in $K/k$, hence there is only one infinite
prime $\pL_{\infty}$ in $K$. Therefore we have 
$U_K=\*{K_{\infty}}\times \prod_{\pK\nmid
\infty} U_{\pK}$ and $U_K^+=\ker \phi_{K_{\infty}}\times \prod_{\pK\nmid
\infty} U_{\pK}$, where $K_{\infty}=K_{\pL_{\infty}}$. Now
\[
K_{\infty}=k_{\infty}(\sqrt[l]{\gamma D})=k_{\infty}(\sqrt[l]{\gamma T^n u})=
k_{\infty}\big(\sqrt[l]{\gamma T^{-r} (T^m)^l u}\big),
\]
where $u=1+a_{n-1}(1/T)+\cdots+a_1(1/T)^{n-1}+a_0(1/T)^n\in U_{K_{
\infty}}^{(1)}$. Since $p\neq l$, there exists $v\in U_{K_{\infty}}^{(1)}$ 
such that $u=v^l$. Therefore $K_{\infty}=k_{\infty}(\sqrt[l]{\gamma T^{-r}})
=k_{\infty}(\sqrt[l]{\delta/T})$ for some $\delta\in\*\F$, $\delta\not\equiv
(-1)\bmod (\*\F)^l$.

A prime element in $K_{\infty}$ is $\Pi:=\pi_{K_{\infty}}=\sqrt[l]{\delta/T}$,
$\Pi^l=\delta/T=\delta\pi_{\infty}$. Hence, for an
arbitrary element $x\in \*{K_{\infty}}$,
$x=\Pi^m \xi w$, $m\in {\mathbb Z}$, $\xi\in\*\F$,
$w\in U_{K_{\infty}}^{(1)}$, we have
$N_{K_{\infty}/k_{\infty}}(x)=((-1)^{l-1}\delta\pi_{\infty})^m
\xi^l w'$, with $w'\in U_{\infty}^{(1)}$, so that $\phi_{K_{\infty}}(x)=
((-1)^{l-1}\delta)^m \xi^l$. It follows that
if $x\in\ker\phi_{K_{\infty}}$, then
$l|m$. We also have
$\delta^{-1}\Pi^l\in \ker\phi_{K_{\infty}}$. 

Therefore
$\min\{m\in{\mathbb N}\mid \text{there exists $\vec\alpha\in U_K^{+}$ with\ }
\deg\vec\alpha=m\}=l$.
It follows the field of constants of $H_K^+$ and of $\ge K$ is ${\mathbb
F}_{q^l}$.

Since the field of constants of $\g K$ is $\F$, we have $\ge K=K_{ge, l}=
\ge E K$.
}}
\end{example}

\begin{example}\label{Ex5.4}{\rm{
Let $l$ be a prime number such that $l^2|q-1$. Let $K=k(\sqrt[l^2]{\gamma D})$
where $\gamma \in \*\F$, $D=P_1^{\alpha_1}\cdots P_r^{\alpha_r}\in R_T$,
$P_1,\ldots, P_r\in R_T^+$, $1\leq \alpha_i\leq l^2-1$, $1\leq i\leq r$, $r\geq 2$.
We assume that $\deg D=ld$ with $l\nmid d$, that $l\nmid \deg P_1$, that
$\gcd(\alpha_i,l)=1$, $1\leq i\leq s$, $s\geq 1$ and that $l|
\alpha_j$, $s+1\leq j\leq r$. Then $e_{P_i}(K/k)=l^2$ for $1\leq i\leq s$,
$e_{P_j}(K/k)=l$ for $s+1\leq j\leq r$, and $e_{\infty}(K/k)=l$.
Since $e_{P_1}(K/k)=l^2$, the field of
constants of $K$ is $\F$. We also assume $(-1)^{\deg D}\gamma \notin
(\*\F)^l$. Therefore $\F(\sqrt[l^2]{\varepsilon})={\mathbb F}_{q^{l^2}}$ where
$\varepsilon=(-1)^{\deg D}\gamma$.

Then $M={\mathbb F}_{q^{l^2}}(T)=k_{l^2}$ and $E=KM\cap \cicl D{}$. Thus
$E=k(\sqrt[l^2]{(-1)^{\deg D}D})=k(\sqrt[l^2]{\*D})\subseteq \cicl D{}$. We have
\[
EK=E(\sqrt[l^2]{\varepsilon})=K(\sqrt[l^2]{\varepsilon})=E_{l^2}=K_{l^2}.
\]
We have $f_{\infty}(K/k)=[\F(\sqrt[l]{\varepsilon}):\F]=l$ and therefore
$f_{\infty}(EK/K)=\frac{f_{\infty}(EK/k)}{f_{\infty}(K/k)}=\frac{l^2}{l}=l$.
Thus $H \cong C_l$ and $|H|=l$.

We also have $\ge E=k\big(\sqrt[l^2]{\*{P_1}},\ldots,\sqrt[l^2]{\*{P_s}},
\sqrt[l]{\*{P_{s+1}}},\ldots \sqrt[l]{\*{P_r}}\big)$. Since $l\nmid \deg P_1$,
it follows that $e_{\infty}\big(k\big(\sqrt[l^2]{\*{P_1}}\big)/k\big)=l^2$ 
and therefore $e_{\infty}(\ge E/
k)=l^2$. Since $\deg D=ld$ with $l\nmid d$ we have $e_{\infty}(\g E/k)=l$.
Now
\[
\deg D=ld =\sum_{i=1}^r \alpha_i\deg P_i =\alpha_i \deg P_1+\sum_{i=2}^s
\alpha_i\deg P_i+\sum_{j=s+1}^r \alpha_j \deg P_j,
\]
and since $l|\sum_{j=s+1}^r\alpha_j\deg P_j$ and $\gcd(\alpha_i,l)=1$
for $1\leq i\leq s$, it follows that there exists $2\leq i\leq s$ with $l\nmid
\deg P_i$. Say $l\nmid \deg P_2$.

Let $a,b\in{\ma Z}$ with $a\deg P_1+bl^2=1$ (in particular, $\gcd (a,l)=1$).
Therefore, we have for $i\geq 2$ that $\deg P_i-(a\deg P_i)\deg P_1=
b(\deg P_i)l^2$. Set $Q_i=P_i P_1^{z_i}$ with $z_i:=-a\deg P_i$ for $2
\leq i\leq r$. Let $L:=k\big(\sqrt[l]{\*{P_1}},\sqrt[l^2]{Q_2},\ldots, 
\sqrt[l^2]{Q_s},\sqrt[l]{Q_{s+1}},\ldots, \sqrt[l]{Q_r}\big)$. Then
$e_{\infty}(L/k)=l$, $e_{P_i}(L/k)=l^2$ for $2\leq i\leq s$ and
$e_{P_j}(L/k)=l$ for $s+1\leq j\leq r$. For $P_1$ we have that since
$l\nmid \deg P_2$, $\sqrt[l^2]{Q_2}=\sqrt[l^2]{P_2P_1^{-a\deg P_2}}$
and $\gcd (l,-a\deg P_2)=1$ so that $e_{P_1}(k(\sqrt[l^2]{Q_2})/k)=l^2$.
Hence $e_{P_1}(L/k)=l^2$. Thus $E\subseteq L$ and $[\ge E:L]=l$, it 
follows that
\begin{gather*}
L=\g E=k\big(\sqrt[l]{\*{P_1}},\sqrt[l^2]{Q_2},\ldots, 
\sqrt[l^2]{Q_s},\sqrt[l]{Q_{s+1}},\ldots, \sqrt[l]{Q_r}\big),\\
\ge E=k\big(\sqrt[l^2]{\*{P_1}},\ldots,\sqrt[l^2]{\*{P_s}},
\sqrt[l]{\*{P_{s+1}}},\ldots \sqrt[l]{\*{P_r}}\big),\\
[\ge E:\g E]=l=e_{\infty}(\ge E/\g E).
\end{gather*}

Now, $EK=K_{l^2}$, so that $H\cong D_{\infty}(EK/K)\cong \Gal(
K_{l^2}/K_{l})$, where $D_{\infty}$ denotes the decomposition group
of the infinite primes.
\[
\xymatrix{
EK=K_{l^2}\ar@{-}[d]^{f_{\infty}(EK/K_l)=l}
\\ K_l\ar@{-}[d]^{f_{\infty}(K_l/K)=1}\\ K
}
\]

We have
\[
E^H=k(\sqrt[l]{\*D})=k(\sqrt[l]{D})\quad \text{and}\quad \g{E^H}
=k\big(\sqrt[l^2]{Q_2},\ldots, 
\sqrt[l^2]{Q_s},\sqrt[l]{Q_{s+1}},\ldots, \sqrt[l]{Q_r}\big).
\]
We also obtain that $\ge{(\g{E^H})}=\ge E$ since $[\ge{(\g{E^H})}:k]=
\prod_{j=1}^re_{P_j}(\g{E^H}/k)=[\ge E:k]$ and $\g{E^H}
\subseteq \ge E$. It follows that 
\begin{gather*}
\ge K=\ge E K=k\big(
\sqrt[l^2]{\*{P_1}},\ldots,\sqrt[l^2]{\*{P_s}},
\sqrt[l]{\*{P_{s+1}}},\ldots \sqrt[l]{\*{P_r}}\sqrt[l^2]{\gamma D}\big)
=\ge E(\sqrt[l^2]{\varepsilon})
\intertext{and}
\g K=\g E^H K=k\big(\sqrt[l^2]{Q_2},\ldots,\sqrt[l^2]{Q_s},\sqrt[l]{Q_{s+1}},
\ldots,\sqrt[l]{Q_r},\sqrt[l^2]{\gamma D}\big).
\end{gather*}
}}
\end{example}

\begin{remark}\label{R5.2}{\rm{
\l
\item When $H=\{1\}$, $\ge K=K^{\ext}=\ge E K$.
\item In general we have $f_{\infty}(\ge K|\g K)>1$
(Examples \ref{Ex5.3} and \ref{Ex5.4}).
\item The field of constants of $H_K^+$ might be different from the one
of $H_K$ (Examples \ref{Ex5.3} and 
\ref{Ex5.4}). In Example \ref{Ex5.4}, ${\ma F}_{q^l}$ is the
field of constants of $H_K$ and ${\ma F}_{q^{l^2}}$ is the field of constants
of $H_K^+$.
\item In general the extension $\ge K/\g K$ might contain ramified 
extensions. In Example \ref{Ex5.4}, $\ge K=\g K(\sqrt[l^2]{P_1})$ and
$\g K\subsetneqq ({\g K})_{l^2}=\g K(\sqrt[l]{P_1})\subsetneqq \ge K$.
We have $e_{\infty}(\ge K|\g K)=f_{\infty}(\ge K|\g K)=l$.
\end{list}
}}
\end{remark}

For the description of $\ge K$ for a finite separable extension $K/k$, at least
in most cases, see \cite{RaRzVi2019}.

\end{document}